\documentclass[reqno, a4paper, draft]{amsart}
\usepackage{amssymb, amscd, amsfonts,  amsthm}
\usepackage[mathscr]{eucal}

\newtheorem{theorem}{Theorem}[section]
\newtheorem{lemma}[theorem]{Lemma}
\newtheorem{corollary}[theorem]{Corollary}

\newtheorem{remark}[theorem]{Remark}
\newtheorem{defi}[theorem]{Definition}

\newcommand{\remove}[1]{}

\DeclareMathOperator{\conv}{\rm conv}
\DeclareMathOperator{\rank}{\mathrm{rank}}
\DeclareMathOperator{\area}{\mathrm{area}}
\DeclareMathOperator{\diameter}{\mathrm{diam}}

\DeclareMathOperator{\linea}{\rm line}
\DeclareMathOperator{\dist}{\rm dist}
\DeclareMathOperator{\relint}{\rm relint}

\title[Two-dimensional diophantine approximation]
{Diophantine approximation by almost equilateral triangles}
\author{Daniele Mundici}
\date{\today}
\address[D. Mundici]{Department of Mathematics and Computer Science
 ``Ulisse Dini'', 
University of Florence,
Viale Morgagni 67,
50134 Florence,
Italy}

\email{mundici@math.unifi.it}

\begin{document}

\begin{abstract}
A {\it two-dimensional 
continued fraction expansion}
is  a map $\mu$  
assigning to every  $x
\in\mathbb R^2\setminus\mathbb Q^2$
  a sequence  
$\mu(x)=T_0,T_1,\dots$  of triangles  
$T_n$ with  
vertices  $x_{ni}=(p_{ni}/d_{ni},q_{ni}/d_{ni})\in\mathbb Q^2,
d_{ni}>0,
p_{ni}, q_{ni}, d_{ni}\in \mathbb Z,$
$i=1,2,3$, such that  
\begin{eqnarray*}
\det
\left(
\begin{matrix}
p_{n1}& q_{n1} &d_{n1}\\
p_{n2}& q_{n2} &d_{n2}\\
p_{n3}& q_{n3} &d_{n3}
\end{matrix}
\right) = \pm 1\,\,\,
\,\,\,\mbox{and}\,\,\,\,\,\,
\bigcap_n T_n = \{x\}.
\end{eqnarray*}
We construct a
two-dimensional 
continued fraction expansion $\mu^*$
such that for densely many (Turing computable) points $x$ 
the vertices of the triangles of $\mu(x)$
strongly converge to $x$. 
Strong convergence depends  on  the value of
$\lim_{n\to \infty}\frac{\sum_{i=1}^3\dist(x,x_{ni})}{(2d_{n1}d_{n2}d_{n3})^{-1/2}},$
(``dist'' denoting euclidean distance)  which in turn depends on the smallest angle of  $T_n$.
Our  proofs combine
a classical  theorem of
Davenport Mahler in diophantine approximation,
with  the algorithmic
resolution of toric  singularities
in the equivalent  framework of regular fans
 and their stellar operations.
\end{abstract}



%
%
%
%
%
%



\maketitle

\section{Introduction and statement of the main results}

%

Following Lagarias \cite{lag}, by a {\it two-dimensional 
continued fraction expansion}
we mean a map $\mu$
assigning to every  $(\alpha,\beta)\in
\mathbb R^2\setminus\mathbb Q^2$
a  sequence  
$\mu(\alpha,\beta)$  of triangles  
$T_n\subseteq \mathbb R^2$ 
with  rational  vertices   
$(p_{ni}/d_{ni},q_{ni}/d_{ni}),\,\,\,
d_{ni}>0,\,$ $p_{ni},q_{ni},d_{ni}\in \mathbb Z,\,\,
\,\,i=1,2,3,$ \,\,$n=0,1,\dots$, such that
$\bigcap_nT_n=\{(\alpha,\beta)\}$ and 
\begin{eqnarray}
\label{equation:unimodular}
\det
\left(
\begin{matrix}
p_{n1}& q_{n1} &d_{n1}\\
p_{n2}& q_{n2} &d_{n2}\\
p_{n3}& q_{n3} &d_{n3}
\end{matrix}
\right) = \pm 1.
\end{eqnarray}

%
 
Triangles 
$T\subseteq \mathbb R^2$ with 
vertices  $(p_i/d_i,q_i/d_i)\in\mathbb Q,$
\,\,\,$d_i>0,\,\,i=1,2,3,$
having 
the unimodularity property 
\eqref{equation:unimodular} are said to be 
{\it regular},  being affine counterparts of  regular 
three-dimensional cones in
$\mathbb R^3,$ \cite[1.2.16]{cox}, \cite[p.146]{ewa}.
Regular cones, in turn, are 
the basic constituents of  regular  fans, i.e.,  
 complexes of regular cones, \cite[3.1.2]{cox},
\cite[V, 4.11]{ewa}.  Regular triangles, cones and fans will find  
use throughout this paper.

For any point  $x=(x_1,x_2)\in\mathbb R^2$ we
let  $G_{x}=\mathbb Zx_1+\mathbb Z x_2,+\mathbb Z$ be the subgroup of the additive group
$\mathbb R$ generated by $x_1,x_2, 1$. Thus
$\rank(G_x)=1$ iff $x\in \mathbb Q^2.$  
Points $x$
with $\rank(G_x)\in\{2\}$ are Lebesgue-negligible,
and their continued fraction  approximation
is a routine variant of 
the classical one.
We will mostly consider points 
$y=(\alpha,\beta)$ with $\rank(G_y)=3,$
called {\it rank 3 points.} Any such point 
 automatically lies in the interior
of every regular triangle containing it.

Suppose a rank 3 point $x$
 is approximated  by a sequence of  triangles  
  $R_0,R_1,\dots$.
Good approximations are in conflict with the
tendency of the   $R_n$ to
degenerate into needle-like triangles
like those of  \cite{beagar,gra,nog2}.
   %
   %
One may naturally wonder whether
 the many negative results on 
strong convergence of 
two-dimensional 
continued fraction expansions, \cite[p.22]{bre}
and  generalized 
Farey se\-q\-uen\-ces, 
\cite[4.1]{gra},
have a geometric counterpart 
in the fact that 
certain scalene  triangles are inevitable. For any triangle $T$
let $\diameter(T)$ be the  length 
of the longest side of $T$.
From a theorem of 
Davenport--Mahler \cite{davmah} 
we have a first partial answer:

\begin{theorem}
\label{theorem:uno}
Infinitely many rank 3 points  $x\in \mathbb R^2$ 
have the property  that
for every  
two-dimensional continued
fraction expansion
$\mu$,  no subsequence 
 $E_0 \supsetneqq E_1 \supsetneqq \dots$ of 
  $\mu(x)=T_0,T_1,\dots$ such that 
   the angles of every  $E_n$ are  $>\arcsin (23^{1/2} / 6)
   \approx   \pi /(3.3921424)\approx 53^\circ$,
  satisfies
 $\bigcap_n E_n = \{x\}$.
Thus, letting for  $i=1,2,3$,
$x_{ni}=(p_{ni}/d_{ni}, 
q_{ni}/d_{ni})$
denote  the $i$th vertex of  $T_{n}$,  
we have  
\begin{equation}
\label{equation:davenport}
\liminf_{n\to \infty}\,\frac{\sum_{i=1}^3\dist(x,x_{ni})}{(2d_{n1}d_{n2}d_{n3})^{-1/2}}
\geq
\liminf_{n\to \infty}\,\frac{\diameter(T_n)}{\area(T_n)^{1/2}}
\geq 
2\cdot\left(\frac{13}{23}\right)^{1/4}\approx 1.734138878.
\end{equation}  
\end{theorem}
 
And yet,  toric resolution of singularities
 \cite{cox}, \cite{ewa}---in
  the equivalent algorithmic-combinatorial framework
\cite{agumun}
of desingularization of fans in $\mathbb R^3$---yields:
\begin{theorem}
\label{theorem:due}
For all $\epsilon>0$ there is 
a two-dimensional continued fraction expansion
$\mu_\epsilon$
having the following property: for densely many
rank 3 points  $x\in \mathbb R^2$, 
the sequence $\mu_\epsilon(x)=E_{\epsilon 0},
E_{\epsilon 1},\dots$ satisfies 
\begin{equation}
\label{equation:epsilon}
\lim_{n\to \infty}\,\frac{\max_{i=1,2,3}\dist(x,v_{ni})}
{(2d_{\epsilon n1}d_{\epsilon n2}d_{\epsilon n3})^{-1/2}}
\leq
\lim_{n\to \infty}\,\frac{\diameter(E_{\epsilon n})}
{\area(E_{\epsilon n})^{1/2}}
< 2\cdot \left(\frac{1}{3}\right)^{1/4}+\epsilon
 \approx  1.52+\epsilon,
\end{equation}
%
%
with  $v_{ni}=(p_{\epsilon ni}/d_{\epsilon ni}, 
q_{\epsilon ni}/d_{\epsilon ni})$
the $i$th vertex of  the triangle $E_{\epsilon n}$,
$(i=1,2,3)$. 
\end{theorem}

\begin{corollary}
\label{corollary:tre}
There is a two-dimensional continued fraction expansion
$\mu^*$ such that for a dense set $\mathcal D$ of 
rank 3 points  $x=(\alpha,\beta)\in \mathbb R^2$,
all angles of every triangle of 
 the sequence $\mu^*(x) =E^*_0,  E^*_1,  \dots$ 
are $>\arcsin (23^{1/2} / 6)$.
Thus,
 \begin{equation}
 \label{equation:corollary}
 \lim_{n\to \infty}\,\frac{\max_{i=1,2,3}\dist(x,v^*_{ni})}
{(2d^*_{n1}d^*_{n2}d^*_{n3})^{-1/2}}
\leq
\lim_{n\to \infty}\,\frac{\diameter(E^*_{n})}
{\area(E^*_{n})^{1/2}}
 <  2\cdot\left(\frac{13}{23}\right)^{1/4},
 \end{equation}
with 
$v^*_{ni}=(p^*_{ni}/d^*_{ni}, q^*_{ni}/d^*_{ni})$
denoting $i$th vertex  of  the triangle $E^*_{n}$,
$(i=1,2,3)$. 
\end{corollary}

\begin{corollary}
\label{corollary:quattro}
With the notation of  Corollary  \ref{corollary:tre}, let
$\rho\subseteq \mathbb R^3$
 be the half-line originating at
 $(0,0,0)$ and passing through $(\alpha,\beta,1)$. For
each $n=0,1,\dots$,  pick a vertex $v^*_n= (p^*_{n}/d^*_{n}, q^*_{n}/d^*_{n})$
of $E^*_n$ having smallest denominator.  Then the sequence
  $v^*_{0}, v^*_{1},\dots$ 
{\em strongly converges to $x$},
in the sense that  
 $\lim_{n\to \infty}\dist(\rho, (p^*_{n},
q^*_{n},d^*_{n}))=0.$ 
\end{corollary}
 Our continued fraction expansions 
 strongly converging over the dense set
$\mathcal D$
of Corollaries \ref{corollary:tre}-\ref{corollary:quattro} 
inherit from the Farey expansion   
 the following  properties:

{\it Approximation steps by Farey sums (=Farey mediants)}
\indent
Each  triangle
 $E^*_{n+1}$
in  $\mu^*(x)$ 
 is obtained from
 $E^*_{n}$   via finitely many  computations of 
 mediants of  pairs of vertices of consecutive triangles.
Thus, passing to homogeneous integer coordinates
 of the vertices of the $E^*_{n}$, , $\mu^*(x)$ is an ``expansion''
 along the ray $\rho$ in the
more restrictive sense of  Brentjes \cite[2.3]{bre} and \cite[p.21]{bre-crelle}.

{\it Turing computability}
  $\mathcal D$ contains a  
  subset $\mathcal D'$ of points, also dense
  in $\mathbb R^2$, 
   which are  the output of an enumerating Turing machine. 
 (See Remark \ref{remark:first}.) Thus for any  $y\in \mathcal D'$
 we have a two-dimensional continued fraction algorithm
 in the sense of \cite[p.21]{bre-crelle}

Farey sums, unimodularity, computability issues,  angles, 
expansions, 
and  the estimates 
\eqref{equation:davenport}-\eqref{equation:corollary}
are not needed for the
   proof of the main result of \cite{mee},
stating that 
 the set 
of points  for which strong convergence fails is
Lebesgue-negligible.

  
Finally, in  Remark \ref{remark:second}, 
our Corollary  \ref{corollary:quattro}
is comparatively discussed  with Grabiner's 
\cite[Theorem 4.1]{gra} stating 
that  for  all $n=2,3,\dots,$ no  
$n$-dimensional Farey continued fraction algorithm
is strongly convergent.  

\section{Proof of Theorem \ref{theorem:uno}}
\label{section:theta}
Following  \cite[p.29]{cox}, a  {\it ray}
$\rho$ in ${\mathbb R}^3$ is a  half-line having the origin  
${\mathbf 0}=(0,0,0)$ as its extremal point. Thus for some
nonzero vector $ {\bf w}= (x,y,z) \in {\mathbb R}^3$ 
we may write
\begin{equation}
\label{equation:monocone}
\rho=\langle {\bf w}\rangle=  
\{\lambda {\bf w}\in \mathbb R^3 \mid  0  \leq\lambda \in {\mathbb R} \}.
\end{equation}
A nonzero integer vector  
${\bf v}= (a,b,c)\in{\mathbb Z}^3\subseteq \mathbb R^3$ is said to be 
{\it primitive} if  $\gcd(a,b,c)=1$, \cite[V, 1.10]{ewa}.
 In other words,
moving  from ${\mathbf 0}$ along the ray  
$\langle \bf v \rangle$, 
$\bf v$ is the first  integer point $\not=\mathbf 0$. 
Following \cite[1.2.14, 1.2.16]{cox} 
by a  {\it rational, three-dimensional,
simplicial  cone}  (``simplex cone'', or ``simple cone'' in \cite[V, 1.8]{ewa})
 in ${\mathbb R}^3$
we mean a set $\sigma\subseteq \mathbb R^3$ of the form  
\begin{equation}
\label{equation:cone}
\sigma = \langle{\bf v}_{1},  {\bf v}_{2},  {\bf v}_{3}\rangle
=
 \left \{\sum_{i=1}^3 \lambda _{i}{\bf v}_{i}
\mid  0\leq \lambda_{i}\in {\mathbb R}\right\},
\end{equation}
for  primitive  
vectors ${\bf v}_{1},{\bf v}_{2},{\bf v}_{3} \in {\mathbb Z}^{3}$ 
whose linear span coincides with ${\mathbb R}^{3}$.
The ${\bf v}_{i}$ are said to be the
{\it primitive  generating vectors} of $\sigma$
(``minimal'' generators, in \cite[p.30]{cox}).
They are uniquely determined by $\sigma$.
(See \cite[p.146]{ewa},  where the notation
${\rm pos}[{\bf v}_{1},  {\bf v}_{2},  {\bf v}_{3}]$
is used  instead of 
$\langle{\bf v}_{1},  {\bf v}_{2},  {\bf v}_{3}\rangle$.
In \cite[1.2.1]{cox} one finds the notation Cone($\{{\bf v}_{1}, 
 {\bf v}_{2},  {\bf v}_{3}\}$).)
 
 Throughout this paper we will  use the following notation:
\begin{equation}
\label{equation:notation}
\mathsf{A} = \{(x,y,z)\in {\mathbb R}^3\mid z = 1\}\,\,\,
{\rm and} \,\,\,
\mathsf{V} = \{(x,y,z)\in {\mathbb R}^3 \mid z > 0\}
\cup\{\mathbf 0\}.
\end{equation}
 
Let  $\rho\subseteq \mathsf V$ be a ray
and  ${\bf v}_{0}, {\bf v}_{1},\dots$   a sequence of 
primitive integer vectors in $\mathsf V$. 
 We then say that  
  ${\bf v}_{0}, {\bf v}_{1},\dots$ 
 {\it strongly converge} to  $\rho $
 if $\lim_{n\to \infty}\dist(\rho,  {\bf v}_{n})=0.$
%
Letting
 $v_n$  be the   orthogonal  projection
 of   $\langle  {\bf v}_{n} \rangle\cap \mathsf A$
  into  $z=0$, and  $r\in \mathbb R^2$  the  
  orthogonal projection of  $\rho\cap \mathsf A$ 
 into  $z=0$, we equivalently  say that the  sequence
 $v_0,v_1,\dots$  {\it strongly
 converges} to $r$, 
  \cite[(4), p.37]{gra}.


Following \cite[1.2.16]{cox} and \cite[V, 1.10]{ewa},
   a rational  cone  
$\sigma = \langle{\bf v}_{1},  {\bf v}_{2},  {\bf v}_{3}\rangle\subseteq
\mathbb R^3$
  is said to be
{\it regular}  $\,\,$ if
$\{{\bf v}_{1},  {\bf v}_{2},  {\bf v}_{3}\}$ 
is a basis of the free abelian group ${\mathbb Z}^3$.
Equivalently, the $3\times 3$ integer
matrix with row vectors ${\bf v}_{1},  {\bf v}_{2},  {\bf v}_{3}$
has determinant $\pm 1$.
%
%
Thus for     $p_i/d_i, q_i/d_i,
d_{i}>0,\, p_{i},q_{i},d_{i}\in \mathbb Z,\,\,
\gcd(p_i,q_i,d_i)=1,
\,\,i=1,2,3,$  the  triangle
$$
T =\conv(({p_1}/{d_1}, {q_1}/{d_1}),({p_2}/{d_2}, 
{q_2}/{d_2}),({p_3}/{d_3}, 
{q_3}/{d_3}))  
$$
is regular iff the primitive integer vectors  ${\bf v}_i=(p_i,q_i,d_i)$ 
generate a regular cone contained in $\mathsf V$.   
If $T$ is regular,  it is easy to see
(\cite[Corollary 11]{beagar}) that 
its area only depends on the least common
denominators  (henceforth, {\it denominators}, \cite[p.465]{lag})
  $d_i$ of the vertices of $T$, 
\begin{equation}
\label{equation:pqr}
\area(T)=\frac{1}{2d_{1}d_{2}d_{3}}.
\end{equation}

 
\noindent
A ray   $\rho = \langle (x,y,z)\rangle\in \mathbb R^3$ is said to be  
  {\it irrational} \,\,if\,\,  $x, y, z$ are linearly
independent over ${\mathbb Q}$. 
Thus the ray
$\langle(\alpha,\beta,1)\rangle$
 is irrational iff 
$(\alpha,\beta)$ is a rank 3 point.
Every  irrational ray  $\rho$ has no integer points
except the origin.
 The  converse is not true, e.g.,  
 for the ray $\langle(x,y,1)\rangle$ whenever
 $(x,y)$ is a  rank 2 point.
 
For  any   
$ 0 \leq \theta <  \pi/3$,  an irrational  ray $\rho \subseteq$  
$\mathsf{V}$    is said to be a
$\theta$-{\it accessible} if  there is 
a  sequence $\sigma_{0} \supsetneqq \sigma_{1}$
$\supsetneqq  \ldots $ of regular three-dimensional cones
 $\sigma_l\subseteq \mathsf{V}$  such that  
$\bigcap_l \sigma_l=\{\rho\}$, 
and for every $n=0,1,\dots,$  the angles of the
triangle $\sigma_{n} \cap \mathsf{A}$ are   
all $ > \theta $.  If $\rho$ is not 
$\theta$-accessible we say that 
$\rho$  is $\theta$-{\it inaccessible.}

\smallskip
As usual, ``conv'' denotes convex hull, 
and ``diam'' is short
for ``diameter''.

\medskip

\begin{lemma}
\label{lemma:ray}
  For every $\theta$ 
satisfying  $\,\, \arcsin (23^{1/2} / 6) <
\theta < \pi/3$  
there is an irrational $\theta$-inaccessible 
 ray $\,\,\rho \subseteq \mathsf{V}.$
\end{lemma}

\begin{proof}
Fix an arbitrary $\theta $ satisfying  
$\,\,\arcsin (23^{1/2} /6)  <  \theta < \pi/3$, and
let
\begin{equation}
\label{equation:citheta}
c_{\theta} = \frac{1}{3 \sin \theta}.
\end{equation}  
Thus  $ \frac{2}{\surd 27} <  c_{\theta} <  \frac{2}{\surd 23}$.
For any such $ c_\theta$, 
Davenport and Mahler   
\cite[Proof of Theorem 
2(b)]{davmah}  exhibit  a pair $(\alpha,\beta)$
of  real
numbers  
satisfying the following condition:

\medskip
\begin{quote}
	$(\ddagger)\,\,$   
	$\alpha,\beta,1\,$ are linearly
independent over $ {\mathbb Q} $ and  there are only
finitely many triplets   $(a, b, c) \in {\mathbb Z}^3$  such
that $0 < a^{2} + b^{2}$   and    
$ \mid a\alpha + b\beta + c \mid $  	    
$ < c_{\theta} / (a^{2}+ b^{2}) $.
\end{quote}
%

\medskip 
\noindent 
Let 
$
\rho= \rho(\theta)= \langle (\alpha, \beta,1) \rangle.
$
By  $(\ddagger)$,
  $\rho\subseteq \mathsf V$ is an irrational ray. 
Arguing by way of contradiction
we will  show that  $\rho$ 
is   $\theta$-inaccessible.


\smallskip
Indeed,  suppose   there exists a   sequence    
$ \sigma_{0}$  
$\supsetneqq\sigma_{1}\supsetneqq \dots$ 
of  regular cones  in $\mathsf{V}$   
such that  $ \bigcap_l \sigma_{l}=\{\rho\}$ and the angles  
of each triangle $\sigma_{l}\cap \mathsf{A}$        
are all $ >\theta $ (absurdum
hypothesis).  
For fixed but otherwise arbitrary  $m = 0, 1,\dots$
let us write for simplicity
$$\sigma = \sigma_{m} =  \langle {\bf u}_{1}, {\bf u}_{2},
{\bf u}_{3} \rangle,$$
with 
$
{\bf u}_{i} = (p_{i}, q_{i}, r_{i}) \mbox{ for suitable integers }
p_{i},q_{i},r_{i} \mbox{ with }  r_{i} > 0,\,\,\, (i=1,2,3).
$         
Let the matrix $M$ be defined by 
\[
M = \left(
\begin{array}{ccc}
p_{1} & q_{1} & r_{1}\\
p_{2} & q_{2} & r_{2}\\
p_{3} & q_{3} & r_{3}
\end{array}
\right).
\]
We  can safely assume  $\det M = 1 $. 
The cone $\sigma$ determines the triangle
$S =  \sigma \cap \mathsf{A}$
with vertices  $(p_i /r_i, q_i /r_i,1),$
\,\,\,$i=1,2,3.$  
Let $S_\downarrow$
be  the orthogonal  projection of $S $ 
to the plane   $z=0$, 
$$
S_\downarrow=\conv(u_1,u_2,u_3)=\conv((p_1 /r_1, q_1/r_1),
(p_2/r_2, q_2/r_2),(p_3 /r_3, q_3 /r_3)).
$$   
The assumed  regularity of
$\sigma$ means that $S_\downarrow$ is
regular.  By \eqref{equation:pqr},  the area of $S$ 
equals 
 $({2r_{1}r_{2}r_{3}})^{-1}.$
 Let  us display the 
  transpose inverse
of  $ M$ by writing 
\begin{equation}
\label{equation:matrix}
L = \left(
\begin{array}{ccc}
a_{1} & b_{1} & c_{1}\\
a_{2} & b_{2} & c_{2}\\
a_{3} & b_{3} & c_{3}
\end{array}
\right).
\end{equation}
By definition, the rows of $ L $ are the
primitive generating vectors of the {\it dual cone} of     
$\sigma$,  \cite[1.2.3]{cox}, \cite[I, 4.1]{ewa}.  
For each $i = 1, 2, 3 $, let  the 
function $ f_i \colon {\mathbb R}^{2}
 \rightarrow  {\mathbb R}$  be defined by     
$f_i(x, y)  = a_i x  + b_i y  +  c_i. $  Then  
$f_i(u_i) = 1/r_i$  and  $f_i$ constantly  
vanishes over the segment $\conv(u_j,u_k)$, 
  where  $(i \not = j,\,\, j\not = k,\,\, k\not = i)$.
Writing for short
$$
\delta_{i} = (a_{i}^2  +  b^2_{i})^{1/2},
$$
it follows that 
\begin{equation}
	\label{equation:2}
	\dist(u_j,u_k) = \frac{\delta_i}{r_{j}  r_k }
	\,\,\,\,\,{\rm and}\,\,\,\,\,
	\dist(u_i,\linea(u_j,u_k)) = \frac{1}{r_i \delta_i}.
	\end{equation}
By our standing absurdum hypothesis,   for each   
$i = 1,2,3 $  the  angle $\theta_i$ of the triangle 
$S_\downarrow$    satisfies the inequalities   
$\theta <\theta_i <  \pi/2.$   Therefore,  
\begin{equation}
	\label{equation:3}
\sin \theta <\sin \theta_i = \frac{r_{i}}{\delta_{j} \delta_k}.
\end{equation}

\medskip
\noindent{\it Claim:}   In the notation
of $(\ddagger)$,
 \eqref{equation:citheta} and \eqref{equation:matrix},
for   at least one
$i \in  \{1, 2, 3\}$ we have  
$\mid  a_{i}\alpha + b_{i}\beta + c_{i}  \mid   
  <   c_{\theta} /( a_{i}^2 + b^2_{i}).$ 

\bigskip
For otherwise, 
 for all $i$  we  have   
$$\,\,\mid a_i \alpha + b_i \beta + c_i \mid\,\,  
=\,\, \delta_i \cdot {\rm dist}((\alpha, \beta), \linea(u_j,u_k)  ) 
\geq  c_\theta / \delta^2_i,$$   
whence by (\ref{equation:2}),  
$$\,\,\,\dist(u_j,u_k)
\cdot {\rm dist}((\alpha, \beta), \linea(u_j,u_k) ) \geq
\frac{c_{\theta}}{\delta^2_i r_j r_k}.$$  
{}From  (\ref{equation:3})  we obtain  
$$\,\,\,2\cdot \area(T_\downarrow)\geq c_{\theta}\cdot 
\sum_{i}\,\frac{1}{\delta_{i}^2 r_{j} r_k}\,.$$
As a consequence,

\medskip
$$\frac{1}{c_\theta} \geq r_1 r_2 r_3\cdot  \sum_i 
\frac{1}{\delta^2_i  r_j  r_k}
= \sum_i  \frac{r_i}{ \delta^2_i} =
\sum_i  \frac{\delta_j  \delta_k  \sin \theta_i}
{\delta_{i}^2}  > \sin \theta\cdot 
\sum_{i}
\frac{\delta_{j} \delta_{k}}{\delta_{i}^2} 
\geq  3 \sin \theta,$$

\medskip
\noindent
which is impossible.
This settles our  claim. 

\bigskip
Let us now turn back to the sequence
$ \sigma_{0}$  
$\supsetneqq\sigma_{1}\supsetneqq \dots$ 
of  regular cones 
introduced at the outset of this proof.
For each  $n=0,1,\dots$ let us now write
$\sigma_n$   =  
$\langle {\bf v}_{1n}, {\bf v}_{2n}, {\bf v}_{3n}\rangle$,  
where,  for each    $i  = 1, 2,3$,
${\bf v}_{i n} = (p_{i n} , q_{i n},r_{i n})$   
$\in  {\mathbb Z}^3 $ and $(p_{i n} /r_{i n},\,\,q_{i n}/r_{i n},\,1)$   
is the  $i$th  vertex of the triangle 
$S_n = \sigma_n \cap \mathsf{A}$. Let   
\[
M_n = \left(
\begin{array}{ccc}
p_{1n} & q_{1n} & r_{1n}\\
p_{2n} & q_{2n} & r_{2n}\\
p_{3n} & q_{3n} & r_{3n}
\end{array}
\right),
\] 

\medskip
\noindent  
and $ L_n$   be the transpose inverse of  $ M_n$.   
Since   $\rho$ is an irrational ray, the
vector $(\alpha,\beta,1)$ lies in the (relative) interior of each
triangle  $S_n$.  Thus from   
$ \bigcap_l \sigma_l=\{\rho\}$
it follows that   for every    $n$    there 
exists    $m$     such that for all  $l   >   m $    no 
row  of  $L_n$ is a row of $L_l$ .  Our claim then
yields  infinitely many   triplets $(a, b, c)$ of integers          
 such that  $ 0  < a^2  +  b^2 $ and  
$\mid a\alpha + b\beta  + c \mid \,\,<\, c_{\theta}/(a^2 + b^2)$.
This contradicts the Davenport-Mahler result  $(\ddagger)$.
\end{proof}

%
%
%

The proof of
Theorem \ref{theorem:uno} immediately follows
from Lemma \ref{lemma:ray} upon setting $x=(\alpha,\beta)$.
The first inequality of
\eqref{equation:davenport}
is a consequence of \eqref{equation:pqr}.

\section{Proof of Theorem \ref{theorem:due} and Corollary
\ref{corollary:tre}}
\label{section:star}

Further
details will be   needed about the set 
of $\theta$-accessible rays in $\mathbb R^3$.
First of all, recalling
 the notational stipulations \eqref{equation:notation},
let us equip the set 
$$
\mathcal R =\{\rho\subseteq \mathsf V\mid \rho  \mbox{ is a ray} \}
$$
 with the topology inherited from the
real projective plane.  Thus a subset  
of $\mathcal R$ is  {\it open}  if it  coincides with
 the set of all rays
intersecting  $U$, for some relatively open set 
$U \subseteq \mathsf A\subseteq \mathbb R^3$.

Next we prepare
the two-dimensional
counterpart of the classical operation
of taking Farey mediants of  segments in
$[0,1]$ with rational vertices. This is frequently
found in  diophantine approximation
(sometimes called ``Farey sum'')
\cite[3.1]{beagar}, 
\cite[2.1]{gra}, 
\cite[p.441]{nog2},
and  is also a main tool for the resolution of
singularities of fans  in 
the theory of toric varieties,  \cite[\S 11.1]{cox}, \cite[III, 2.1]{ewa}.
It will find pervasive use in this section.
%
 
 As in  \eqref{equation:monocone} and
  \eqref{equation:cone},
given a cone
$\psi = \langle {\bf u}_{1},{\bf u}_{2},{\bf u}_{3}\rangle \in {\mathbb R}^{3}$
and $i\not=j\in\{1,2,3\}$  we use the notation
 $$
 \langle {\bf u}_i,{\bf u}_j\rangle = \{\alpha  {\bf u}_i+
\beta  {\bf u}_j \in \mathbb R^3\mid 0\leq  \alpha,\beta\in \mathbb R \}.
 $$
The two-dimensional  cones   
$\langle {\bf u}_{1},{\bf u}_{2}\rangle,
\langle {\bf u}_{3},{\bf u}_{1}\rangle$,  
and $\langle {\bf u}_{2},{\bf u}_{3}\rangle$ are   the 
two-dimensional {\it faces} of $\psi$.
The vector ${\bf u}_{1}+ {\bf u}_{2} \in {\mathbb Z}^3$
is called the (Farey) {\it mediant} 
of ${\bf u}_{1}$ and ${\bf u}_{2}$. 
In case  $\psi$ is regular,
    ${\bf u}_{1}+ {\bf u}_{2}$ is primitive;    the  three-dimensional simplicial cones  
$\langle {\bf u}_{1}+ {\bf u}_{2} ,{\bf u}_{2},{\bf u}_{3}\rangle$
and $\langle {\bf u}_{1}+ {\bf u}_{2} ,{\bf u}_{1},{\bf u}_{3}\rangle$      
are said to be obtained by the {\it binary starring}  of  
$\psi$   at  ${\bf u}_{1}+ {\bf u}_{2}$.
We write, respectively,   
\begin{equation}
\label{equation:binary}
 \langle {\bf u}_{1},{\bf u}_{2},{\bf u}_{3}\rangle
 \mapsto^*
 \langle {\bf u}_{1}+ {\bf u}_{2} ,{\bf u}_{2},{\bf u}_{3}\rangle
 \mbox{ and }
  \langle {\bf u}_{1},{\bf u}_{2},{\bf u}_{3}\rangle
 \mapsto^*
 \langle {\bf u}_{1}+ {\bf u}_{2} ,{\bf u}_{1},{\bf u}_{3}\rangle. 
\end{equation}
Both cones 
$\langle {\bf u}_{1}+ {\bf u}_{2} ,{\bf u}_{2},{\bf u}_{3}\rangle$
and $\langle {\bf u}_{1}+ {\bf u}_{2} ,{\bf u}_{3},{\bf u}_{1}\rangle$ are regular.

As usual,
a  {\it simplicial complex} in $\mathbb R^n$ is a finite
set $\mathcal K$ of simplexes in  $\mathbb R^n$ , closed under taking faces,
and having the further property that any two elements of 
$\mathcal K$ intersect in a common face,
 \cite[I, p.66]{ewa}.
For every simplicial complex  $\Sigma$, the
point set union of the simplexes of $\Sigma$ is
called the {\it support} of $\Sigma$, denoted
$|\Sigma|$, \cite[I, 1.9]{ewa}. 
As a special case of a general
definition, when $|\Sigma|\subseteq \mathbb R^2$
coincides with the point set 
 union of the triangles in  $\Sigma$,
we say that the simplicial complex 
 $\Sigma$ is {\it regular}
 (``unimodular'' in \cite{mun-dcds}) if so are all its triangles.
 We also say that $\Sigma$ is a
 regular {\it triangulation} of 
 its support.  Regular triangulations
are  affine  counterparts 
of regular fans, \cite[3.1.18]{cox},  \cite[V, 4.11]{ewa}.
 If $\Sigma_1$ and $\Sigma_2$ have the same
{support}   
and every simplex of $\Sigma_1$ is contained
in some simplex of $\Sigma_2$, we say that
$\Sigma_1$ is a {\it subdivision} of $\Sigma_2.$

\smallskip
For any triangle
$\conv(P,Q,R)$,  let
 $P\widehat Q R$ denote the angle with vertex $Q$.
By a traditional abuse of notation,
we will use the same notation for angles
 and their measure.

For the proof of the next lemma we let
$$
{\bf n}_{0} =  (p_{0}, q_0, r_{0}), \,\,\,
 {\bf n}_{1} = (p_{1}, q_{1}, r_{1}), \dots
$$
enumerate  (in some prescribed lexicographic order)
 the totality of
 primitive vectors  
 ${\bf n} = (p,q,r) \in {\mathbb Z}^3$   
satisfying the inequality 
 $\mid p \mid + \mid q \mid \,\,\, > 0.$
 For each $i=0,1,\dots$ let
 the plane  ${\bf n}_{i}^\perp\subseteq \mathbb R^3$
 be defined by 
 $ p_i x + q_i y +r_{i} z = 0$.
 It follows that    ${\bf n}_{i}\cap \mathsf A$ is a line, denoted 
 $\Lambda_{i}$.
 Moreover, the sequence 
 \begin{equation}
 \label{equation:lines} 
 \Lambda_0,\,\,\,\Lambda_1,\dots
 \end{equation}
gives  all possible {\it rational} lines
 lying on $\mathsf A,$   (i.e., all  lines
 containing at least two distinct rational points
 of $\mathsf A$).

\begin{lemma}
\label{lemma:second-bis}
For any  
	$ \theta  < \pi/3$
	the set of irrational $\theta$-accessible rays
	 is dense in $\mathcal R$.
	\end{lemma}
	
	\begin{proof}
It suffices to prove the lemma
under
the more restrictive condition
\begin{equation}
\label{equation:pifourth}
  \pi/4 <\theta <\pi/3.
\end{equation}	
	The set  $\mho$ of all relatively open right
	triangles with rational vertices 
	lying on the plane  $\mathsf A$
	is a basis of the natural  
	topology of $\mathsf A$ inherited by
	restriction from the usual  topology of
	$\mathbb R^3$.
	So
	for any  nonempty
$\mathcal O\in \mho$  
  we must  show that   
 some $\theta$-accessible ray 
 has a 
nonempty intersection with
$\mathcal O$.  

 Let $\overline{\mathcal O}$ denote the closure of $\mathcal O.$
There exists a uniquely determined
 (necessarily rational, 
simplicial, three-dimensional) cone 
$\tau \subseteq \mathsf{V}$ such that the     
triangle $\tau \cap \mathsf{A}$ coincides with
 $\overline{\mathcal O}$.

 Toric resolution of singularities, \cite[\S 11.1, p.113]{cox},
 \cite[VI, proof of 8.5, p.165]{ewa}, 
 (whose fan-theoretic reformulation for $\tau$
  amounts to starring at  
  primitive integer vectors arising from
   iterated applications of
Blichtfeldt's theorem in
the Geometry of Numbers, \cite[9, p.35]{lek},
 \cite[1.2]{mun-adv}, \cite[p.544]{mun-dcds})
%
%
%
%
%
yields a {\it regular fan over  $\tau$},
i.e., a complex $\Delta$ of
regular cones and their faces,   such that  $\tau$ coincides  with
the point set union of the cones of $\Delta.$
Interestingly enough,   from the input data
consisting of the vertices of
 $\overline{\mathcal O}$,  a regular
fan  $\Delta$ can be effectively  computed 
having special minimality properties, \cite{agumun},
which are characteristic of desingularizations of fans
in $\mathbb R^3$,
and are reminiscent of the Hirzebruch-Jung
continued fraction algorithm for  the smallest
resolution of singularities of fans in $\mathbb R^2$,
 \cite{cox}.

\smallskip
Let  $\tau =  \langle {\bf a},{\bf b},{\bf c} \rangle\in \Delta$
be a  three-dimensional  cone with  $\tau\cap \mathsf A=\conv(A,B,C)$.
Then without loss of generality we can write
 $$
\pi/3< A\widehat{C}B,\,\,\,\,\,\,\,\,\,\,\, 
A\widehat{C}B\geq C\widehat{A}B
 \geq A\widehat{B}C,\,\,\,\,\,\,\,\,\,
 A\widehat{B}C< \pi/3. 
$$   
  

\medskip
\noindent
{\it Preamble.}
In case
$\conv(A,B)\subseteq \Lambda_0$, via one binary 
 starring 
we replace  $\tau$ by
$\tau_0= \langle {\bf a},{\bf b}+{\bf c},{\bf c}\rangle$,  then
give new names ${\bf a}_0,{\bf b}_0,{\bf c}_0$
to the primitive generating vectors of $\tau_0$
and write  $A_0,B_0,C_0$
for  the corresponding vertices of the triangle
$\tau_0\cap\mathsf A$ in such
a way that

\begin{equation}
\label{equation:c}
\pi/3< A_0\widehat{C}_0B_0,\,\,\, 
A_0\widehat{C}_0B_0\geq C_0\widehat{A}_0B_0
 \geq A_0\widehat{B}_0C_0,\,\,\,
 A_0\widehat{B}_0C_0< \pi/3. 
\end{equation}%

\medskip
\noindent
In case $\conv(A,B)\nsubseteq \Lambda_0$,  we just set 
$\tau_0= \langle {\bf a},{\bf b},{\bf c}
  \rangle= \langle {\bf a}_0,{\bf b}_0,{\bf c}_0
  \rangle$  and let 
  $$
  A_0=\langle {\bf a}_0 \rangle\cap\mathsf A,\,\,\,\,\,
    B_0=\langle {\bf b}_0 \rangle\cap\mathsf A,\,\,\,\,\,
      C_0=\langle {\bf c}_0 \rangle\cap\mathsf A.
  $$
In any case, the regular cone $\tau_0$ satisfies
\eqref{equation:c}  as well as
\begin{equation}
\label{equation:sfuggita}
\conv(A_0,B_0)\nsubseteq \Lambda_0.
\end{equation}

Next we proceed with the following two steps:

\medskip
\noindent {\it Step 1.}
Let   the point $J $ be defined by
$
J\in \conv(A_0,B_0),\,\,\, A_0\widehat{J}C_0 =\pi/3.
$
The existence   of $J$ is ensured by
 \eqref{equation:c}.  By 
\eqref{equation:pifourth}
 and
  \eqref{equation:sfuggita},
 there exists a point   $I$ satisfying  the following
 conditions:
 \begin{equation}
 \label{equation:i}
 I\in\conv(A_0,B_0),\,\,\, I\notin \mathbb Q^2, \,\,\,I\notin\Lambda_0, \,\,\,
 \frac{\pi-3\theta}{6} < A_0\widehat{I}C_0 < \frac{\pi}{3}.
 \end{equation}
 Since $I$ is not a rational point,
 the  ray $\rho$ through $I$ is contained in exactly one
  of the two
  three-dimensional cones obtained by binary starring 
  $\tau_{0}$ at ${\bf a}+{\bf b}$.
Letting  $\tau_1$ denote such  cone, we may write
in more detail
$$ \tau_1
= \langle {\bf a}_{1},{\bf b}_{1},{\bf c}_0\rangle \in 
\{\langle {\bf a}_0, {\bf a}_0+{\bf b_0}, {\bf c}_0\rangle,
\langle {\bf b}_0, {\bf a}_0+{\bf b}_0,{\bf c}_0\rangle \}.$$
Keeping   ${\bf c}_0$ fixed and proceeding as in the
classical slow continued fraction algorithm,  we have a    
sequence of  regular three-dimensional cones
containing $\rho\ni I,$ 
$$
\tau_{0} \supsetneqq  \tau_{1}  \supsetneqq \dots,
$$
where 
$\tau_{n} = \langle {\bf a}_{n},{\bf b}_{n},{\bf c}_0 \rangle\ni\rho$,
and  $\tau_{n+1}$ is the result of a binary starring
of $\tau_n$.  The sequence does not  terminate, because 
$I\notin \mathbb Q^2.$ Let us write 
$$\tau_n\cap\mathsf A=
\conv(A_n,B_n,C_0),
\mbox{ where } A_n=  \langle {\bf a}_n \rangle\cap\mathsf A
\mbox{ and } 
B_n=  \langle {\bf b}_n \rangle\cap\mathsf A.$$
 Since   $B_n\widehat{C}_0A_n$ shrinks to zero
as $n$ tends to $\infty$, then 
 $\bigcap_l \langle {\bf a}_{l},{\bf b}_{l}\rangle=\{\rho\}$.
 Further, 
$
\lim_n A_n\widehat{B_n}C_0=A_0\widehat{I}C_0.
$
Since by
\eqref{equation:i},
   $I\notin \Lambda_0$, then 
   for some $\zeta\in \mathbb R$   and  $0<m\in\mathbb Z$
the triangle   $\conv(A_m,B_m,C_0)$
satisfies
\begin{equation}
\label{equation:partialcond(i)}
B_m \notin  \Lambda_0,\,\,\,\,\,\,\,\,\,
\theta < \zeta = A_{m}\widehat{B_{m}} C_0<\pi/3,\,\,\,\,\,\,\,\,
\{A_0,B_0\}\cap\{A_m,B_m\}=\emptyset.
\end{equation}  
Step 1 is completed.

\bigskip
\noindent {\it Step 2.}  
Next we keep fixed  the primitive generator ${\bf b}_m$ of
 $\tau_m= \langle {\bf a}_{m},{\bf b}_{m},{\bf c}_0 \rangle$,
 and proceed as in Brentjes' \cite[2.3(2), p.19]{bre}.
For all   $0 \leq  p,q \in {\mathbb Z} $,  the cone
$\sigma_{p,q} = \langle {\bf a}_m + p{\bf b}_{m},\,\,\,
{\bf c}_{m} + q{\bf b}_{m},\,\,\, {\bf b}_{m}\rangle$
 is regular and is   contained in $ \tau_{m}$. 
 Further, $\sigma_{p,q}$ is obtainable from
 $ \tau_{m}$ by a   sequence of $p+q$ binary starrings.
 (See \cite[Fig 3, p.58]{gra} for an illustration of these 
 binary starrings, with the caveat that  
 all vertices therein have different names from those
 of our triangles here.)
Let  
$$
A_{p,q}=\langle {\bf a}_m + p{\bf b}_{m}\rangle\cap \mathsf A,\,\,\,\,
C_{p,q}=\langle{\bf c}_{m} + q{\bf b}_{m}\rangle\cap\mathsf A.
$$


For all $p,q$   
$ A_{p,q}\widehat{B_{m}}C_{p,q}$ is  equal to $ \zeta$. 
Further,  $\bigcap_{p,q}\sigma_{p,q}
= \{\langle {\bf b}_{m} \rangle\}$. 
Next let 
\begin{itemize}
\item[] 
$ \Pi_{p,q}$  be the plane in $\mathbb R^3$ determined by  the
three points  ${\mathbf 0},  A_{p,q}, C_{p,q}$. 

\smallskip 
\item[] $\Lambda_{p,q}$ be the rational line    
  $ \Pi_{p,q} \cap  \mathsf{A}$.
  
  \smallskip 
\item[]   $\Omega_{p,q}\subseteq \mathsf A$  
be  the bisector of the angle
 $A_{p,q}\widehat{B_{m}}C_{p,q}$ 
  of the triangle $\sigma_{p,q} \cap \mathsf{A}$.
\end{itemize}
%
%

%
%
%
%
%
\noindent
For  infinitely many pairs $(p',q')$ of integers  $>0$,
 the   vertical angles of the two intersecting
lines $\Lambda_{p',q'}$ and $\Omega_{p',q'}$
can be made arbitrarily close  to  $\pi/2$. 
Correspondingly,    
\begin{equation}
\label{equation:iv}
\mbox{
$B_{m} \widehat{C_{p',q'} }A_{p',q'} $   and  
$B_{m}\widehat{A_{p',q'}}C_{p',q'} $ get arbitrarily 
close to $(\pi - \zeta)/2  >  \pi/3$.}
\end{equation}
%
%

\medskip
\noindent
Thus  by 
 \eqref{equation:partialcond(i)},
for (infinitely many
pairs of)    integers  $p',q' >0$, 
the  regular cone 
$\sigma_{p',q'} = \langle {\bf a}_m + p'{\bf b}_{m},\,\,\,
{\bf c}_{m} + q'{\bf b}_{m},\,\,\, {\bf b}_{m}\rangle$
intersects $\mathsf A$ in a triangle
having all angles $ > \theta $, and also satisfying
\begin{equation}
\label{equation:conditioni}
\sigma_{p',q'} \cap \Lambda_0 =\emptyset
\,\,\,\mbox{and}\,\,\,\, 
\{A_0,B_0,C_0\}\cap\{A_{p',q'},B_m, C_{p',q'}\}=\emptyset.
\end{equation}
Letting
``relint'' denote relative interior, we automatically have
\begin{equation}
\label{equation:relint}
 \conv(A_{p',q'},B_m, C_{p',q'})\subsetneqq \relint\conv(A_0,B_0,C_0)
 \subseteq \mathcal O.
\end{equation}
We finally introduce
 the following  notational abbreviations: 
$$
 \tau_{n_0}=\sigma_{p',q'}= \langle {\bf a}_m + p'{\bf b}_{m},
 \,\,\, {\bf b}_{m},\,\,\,
{\bf c}_{m} + q'{\bf b}_{m}\rangle =
\langle{\bf a}_{n_0},{\bf b}_{n_0}, {\bf c}_{n_0}\rangle
$$
and
$$
 \tau_{n_0}\cap \mathsf A=
 \conv(A_{n_0},B_{n_0},C_{n_0}),
$$
where
$$
A_{n_0}=\langle {\bf a}_{n_0}\rangle\cap\mathsf A,\,\,\,\,\,
B_{n_0}=\langle {\bf b}_{n_0}\rangle\cap\mathsf A,\,\,\,\,\,
C_{n_0}=\langle {\bf c}_{n_0}\rangle\cap\mathsf A.
$$
Now Step 2 is completed.

\bigskip
\noindent
The finite
path of   binary starrings
$$ \tau_0\mapsto^*\tau_{1}\mapsto^*\dots\mapsto^*\tau_m \mapsto^*
\tau_{m+1}\mapsto^*\dots\mapsto^*
 \tau_{n_0}
 $$
results in  a new  regular cone  $ \tau_{n_0}$
 having the following properties:

\medskip
 \begin{itemize}
 \item[(i)]   $\conv(A_{n_0},B_{n_0},C_{n_0})\cap \Lambda_0
 =\emptyset,$ 
  \,\,\,  by \eqref{equation:conditioni}.

 \smallskip
  \item[(ii)]  $\conv(A_{n_0},B_{n_0},C_{n_0})\subseteq
\relint\conv(A_0,B_0,C_0)$,
 \,\,\,  by \eqref{equation:relint}.

 \smallskip
   \item[(iii)] $\theta < A_{n_0}\widehat{B_{n_0}} C_{n_0}<\pi/3,$
 \,\,\,  by
\eqref{equation:partialcond(i)}. 
 
  \smallskip
    \item[(iv)]
    $B_{n_0}\widehat{A_{n_0}} C_{n_0},$
       $B_{n_0}\widehat{C_{n_0}} A_{n_0}>\pi/3>\theta,$
 \,\,\,  by \eqref{equation:iv}.       
\end{itemize}

\bigskip

\noindent {\it Claim 1.}  
Suppose we are given a regular cone  $\tau_{n_k}$
such that  the triangle 
 $\tau_{n_k}\cap \mathsf A$ is disjoint from
  $\Lambda_1\cup\dots\cup\Lambda_k,$
  has an angle  $\zeta_{n_k}$ with   $\theta<\zeta_{n_k}< \pi/3$, and
has the other two angles  $>\pi/3.$
  Then  a finite path of  binary starrings
  produces from  $\tau_{n_k}$  a regular cone
  $\tau_{n_{k+1}}$ such that the triangle
 $\tau_{n_{k+1}}\cap \mathsf A$
  is disjoint from  $\Lambda_1\cup\dots\cup\Lambda_{k+1},$
is strictly contained in the relative interior of
 $\tau_{n_k}\cap \mathsf A$,
has an angle $\zeta_{n_{k+1}}$ with   $\theta<
\zeta_{n_{k+1}}< \pi/3$, 
and has the other
two angles  $>\pi/3.$

 \smallskip
The proof is by induction on $k=0,1,\dots$, following Steps 1 and 2
with $\tau_{n_k}$ in place of $\tau_0$.
 Since  $\tau_{n_{k+1}}\subseteq \tau_{n_{k}}$, to ensure that
$\tau_{n_{k+1}}\cap \mathsf A$ is disjoint
from
 $\Lambda_1\cup\dots\cup\Lambda_{k+1},$
 it is sufficient to guarantee that it is
 disjoint from $\Lambda_{k+1}$.
(In case  $\Lambda_{k+1}$ contains
 the largest side of the triangle
   $\langle \tau_{n_{k+1}} \rangle\cap \mathsf A$
we perform a preliminary binary 
starring of  $\tau_{n_{k+1}}$ 
as in the preamble above, before taking Steps 1 and 2.)

\bigskip

  

\medskip
\noindent
Having proved our claim,
let us  fix the following notation,
for all $n,k=0,1,\dots$: 
\begin{equation}
\label{equation:t}
T_n =\mbox{ orthogonal 
projection of $\tau_n\cap\mathsf A$ into the plane $z=0$.}
\end{equation}    
\begin{equation}
\label{equation:etad}
 \eta_k =\tau_{n_k}, \,\,\,\,
D_k=\eta_k\cap \mathsf A,\,\,\,\,
E_n=  \mbox{ orthogonal
projection of $D_n$ into  $z=0$.}
\end{equation}
For notational simplicity,
 the dependence on $\theta$ and $\mathcal O$ of
$\tau_n, \eta_k, T_n, D_n, E_n$ is tacitly understood.



\bigskip
\noindent {\it Claim 2.}   
 $\bigcap_l\eta_l$ is a singleton consisting of a
$\theta$-accessible ray  $\rho=\rho_{\mathcal O,\theta}$  
such that the point $\rho \cap \mathsf A$
lies in $\mathcal O$, and the orthogonal projection of 
$\rho \cap \mathsf A$ into the plane $z=0$ is a
rank 3 point.

\smallskip
As a matter of fact,  
 the compactness of each triangle
$\tau_n\cap \mathsf A$ ensures that 
$\bigcap_l\eta_l$ is nonempty.
 For $i=1,2,3$ let
$(p_{ni}/d_{ni}, q_{ni}/d_{ni})$ be the vertices of $T_n$.
Recalling
\eqref{equation:pqr},
from the  regularity of $\tau_n$ it follows that 
$\area(T_n)=1/(2d_{n1}d_{n2}d_{n3})$.
After
each binary  starring
$\tau_j\mapsto^* \tau_{j+1}$ as in
 \eqref{equation:binary},  two
vertices of $T_j$  are also vertices of
$T_{j+1}$.
The denominator of the third vertex
of $T_{j+1}$ is strictly greater
than the denominator of the third vertex of
$T_j.$  Thus
$\area(T_{l}) \downarrow   0.$   
Since the  $E_i$ are a subsequence of the
 $T_j$, it follows that 
$\area(E_{l}) = \area(D_{l}) \downarrow   0.$   
Since by  \eqref{equation:pifourth}
 for every  $n$ all  angles of  
  $D_{n}$, are $>\theta>\pi/4$, elementary
  geometry shows that $\bigcap_lD_l$ is
  a singleton point.
  (With reference to our  remarks in
  the Introduction, no tendency here is possible for
  the $D_n$   to      
degenerate into needle-like triangles---because
the area of $D_n$ controls its diameter.)
Thus $\bigcap_l\eta_l$ is a singleton ray, which
has the desired properties, by Claim 1 and 
\eqref{equation:relint}. In particular, the orthogonal projection of 
$\rho \cap \mathsf A$ into the plane $z=0$ is a
rank 3 point because $\rho \cap \mathsf A$ lies in
no rational line  $\Lambda_n.$ 

%
%
%

 \medskip
Having thus settled Claim 2,  the proof of 
Lemma \ref{lemma:second-bis} is complete.
\end{proof} 

\begin{remark}
\label{remark:first}
{\rm  Perusal of the proof  
of Lemma \ref{lemma:second-bis}
shows that, once
a fixed finite alphabet
$\mathcal A$  is chosen, and strings over $\mathcal A$
representing 
 rational points,  triangles,  and cones 
are equipped with some prescribed lexicographic
order,  then  for every  $\theta<\pi/3$ 
{\it there exists}  a Turing machine  $\mathcal M_\theta$
having the following property:

\smallskip
\begin{quote}
Over any  input  integer $n\geq 0$ 
together with the vertices  of  a  triangle $\overline{\mathcal O}$,
for  $\emptyset \not= \mathcal O\in \mho,$
$\mathcal M_\theta$
outputs the $n$th term $\eta_n$ of a 
sequence  $\eta_0\supsetneqq \eta_1
\supsetneqq,\dots$ of regular cones 
 in $\mathsf V$
 closing down to an irrational
  ray  
which intersects $\mathsf A$ at a point $y$ of $\mathcal O$,
and has the additional property that 
all angles of every
triangle  $\eta_n\cap \mathsf A$ are $>\theta.$ 
\end{quote}

\smallskip
\noindent
Thus the projection of $y$
into the plane $z=0$ is  a  pair of recursively
enumerable real numbers.
In particular, if $\sin(\theta)$ is  rational,
the instructions/quintuples
of  $\mathcal M_\theta$ can be {\it effectively}
written down, following the constructive proof of
 Lemma \ref{lemma:second-bis}.
}
\end{remark}




\section{Proofs of 
 Theorem \ref{theorem:due}, Corollary \ref{corollary:tre} and
 Corollary \ref{corollary:quattro}}
\subsection*{Proof
 of Theorem  \ref{theorem:due}}
Recall the definition of the basis 
$\mho$  given  at the beginning of the proof of
 Lemma \ref{lemma:second-bis}.  
For every nonempty $\mathcal O\in \mho$
and angle  $\theta<\pi/3$,  the proof
constructs 
the  irrational  $\theta$-accessible  ray  
$ \rho_{\mathcal O,\theta}\in \mathcal O$.
Let the
 rank 3 point   $x_{\mathcal O,\theta}\in \mathbb R^2$
 be defined  by 
 $$
 x_{\mathcal O,\theta} =\mbox{
 orthogonal projection of }
 \rho_{\mathcal O,\theta}\cap \mathsf A
 \mbox{ into the plane } z=0.
 $$ 
Let further
$$
\mathcal D_\theta=\{x_{\mathcal O,\theta}\in \mathbb R^2 \mid
\mbox{ for some nonempty } {\mathcal O}\in \mho \}.
$$
By Lemma \ref{lemma:second-bis},
   $\mathcal D_\theta$ is a dense set of rank 3
   points  in $\mathbb R^2$.
   Let   the sequence   
   $ E_{0,\mathcal O, \theta},E_{1,\mathcal O, \theta}, \dots$ be as  
in \eqref{equation:etad}---the 
 dependence on
   $\theta$ and $ \mathcal O$  being now  made explicit.
Let   the
 two-dimensional continued fraction
 expansion
 $\mu_\theta$ be defined as follows:
 \begin{itemize}
 \item[(I)] For each  $x=x_{\mathcal O,\theta}\in \mathcal D_\theta$,
 $\mu_\theta(x)$ is the sequence    
 $ E_{0,\mathcal O, \theta},E_{1,\mathcal O, \theta}, \dots$.
 
 \smallskip  
 \item[(II)] For all other pairs $y\in \mathbb R^2$ of irrationals, 
 $\mu_\theta(y)$ is, e.g.,   the 
 two-dimensional  continued 
 fraction expansion in  \cite{bre-crelle}.
\end{itemize}
For  every  $\epsilon>0$ there is  
 $\theta=\theta(\epsilon)$ so close
 to $\pi/3$ that  the
 two-dimensional continued fraction
 expansion  $\mu_\epsilon=\mu_{\theta(\epsilon)}$ 
 satisfies condition  \eqref{equation:epsilon}
 in Theorem  \ref{theorem:due}. As a matter of fact, 
on the one hand, for all 
   equilateral triangles, the ratio
 between diameter (=side length) and square root
 of area is $2\cdot 3^{-1/4}$. On the other hand,
since  all  angles of $E_{n,\mathcal O, \theta}$ are $>\theta$,
an elementary geometric
argument  ensures that  the  ratio
 between the diameter of each  $E_{n,\mathcal O,\theta}$  
and the square root
 of the area of $E_{n,\mathcal O,\theta}$ is   $\leq 2\cdot 3^{-1/4}+k_\theta,$
 where the constant $k_\theta$ is independent
 of $n$ and $\mathcal O,$ and  tends to zero as
 $\theta$ tends to $\pi/3$ from below.
 
 The proof  of Theorem \ref{theorem:due} is complete.

\subsection*{Proof
 of Corollary  \ref{corollary:tre}}
The proof of the first statement in 
 Corollary \ref{corollary:tre} is a routine variant of the
 proof of  Theorem \ref{theorem:due},
arguing for the special case   $\theta^* =
\arcsin (23^{1/2} / 6)$
 in Lemma \ref{lemma:second-bis},
 and defining  $\mu^*=\mu_{\theta^*}$
 as in (I)-(II) above. 
To verify   \eqref{equation:corollary},
let  $E$ range over the totality $\mathcal E$ of 
triangles whose angles are 
  $\geq \arcsin (23^{1/2} / 6)$.
Then elementary geometry shows that
  the ratio between the diameter of $E$ and
the  square root of the area
of $E$  attains the maximum value
$2\cdot(13/23)^{1/4}$ when $E$ is the isosceles
 triangle with two equal angles of
$\arcsin (23^{1/2} / 6)$ radians; 
for  all other triangles in $\mathcal E$ this
ratio is $<2\cdot(13/23)^{1/4}$. 
The first identity in
\eqref{equation:corollary}
follows from  \eqref{equation:pqr}. 
 
The proof of Corollary  \ref{corollary:tre} is complete.

\bigskip
\subsection*{Proof
 of Corollary  \ref{corollary:quattro}}
Let $\mathsf {k}=2\cdot({13}/{23})^{1/4}$. For all $n
=0,1,\dots$ we can write
without loss of generality
$$
 v^*_n=v^*_{n,1}= (p^*_{n,1}/d^*_{n,1}, q^*_{n,1}/d^*_{n,1}),  \,\,d^*_{n,1}\leq
 d^*_{n,2}\leq d^*_{n,3}.
$$
By \eqref{equation:corollary}
in Corollary
\ref{corollary:tre} we have
\begin{eqnarray*}
\dist(\rho, v^*_{n,1}) &<& d^*_{n,1}\cdot\diameter(E^*_n)\\[0.1cm]
&<& \mathsf {k}\cdot d^*_{n,1} (2d^*_{n,1}d^*_{n,2}d^*_{n,3})^{-1/2}\\[0.1cm] 
&=& \frac{\mathsf {k}(d^*_{n,1})^{1/2}}{(2d^*_{n,2}d^*_{n,3})^{1/2}}\\[0.1cm] 
&\leq& \frac{\mathsf {k}}{(2d^*_{n,3})^{1/2}}.
\end{eqnarray*}
By construction,  $\lim_{n\to\infty}d^*_{n,3}\to \infty$, whence
$\lim_{n\to\infty}\dist(\Lambda, v^*_{n,1})=0$, as desired
to prove that the sequence of primitive integer vectors
$(p^*_{n,1},q^*_{n,1},d^*_{n,1})$ strongly converges to
the ray
$\rho$  through $(\alpha,\beta,1)$.
Correspondingly, the vertices  $v^*_{0}, v^*_{1},\dots$ 
strongly converge to the rank 3 point $(\alpha,\beta)=x$. 

\begin{remark}
\label{remark:second}
{\rm It is instructive enough to 
compare Corollary  \ref{corollary:quattro}  with
\cite[Theorem 4.1]{gra} stating 
that  for  all $n=2,3,\dots,$ no  $n$-dimensional Farey continued fraction algorithm is strongly convergent. 

Both in our approach here  and in \cite{gra}
the unimodular property, 
the effective
computability of approximating sequences
 of points,  the rank of points, their denominators,   
Farey mediants, 
  stellar operations
  and two-dimensional 
continued fraction expansions have a basic role.  

However,  our algorithmic approach is  local:
we study the geometric properties of
sequences of  
triangles closing down to {\it a single    point}  $x\in \mathbb R^2$, just as 
 the classical   continued
fraction algorithm does by providing a
 sequence of rational segments 
 whose vertices converge to
  a single point  $y\in [0,1]$.
  

By contrast, 
following the time-honored tradition of 
\cite[\S 8]{hur}  
and many other papers (see \cite{nog2} and
references therein),
the subject matter of \cite{gra} 
 is the generalization of 
the Farey sequence and its variants,
\cite{ste}. 
So these papers deal with
two-dimensional 
continued fraction expansions
arising from  
sequences of finer and finer triangulations
$\nabla_n$
of a given fixed  domain $D$,  in such a way that
the denominator of each vertex of
each triangle of $\nabla_n$  is $\leq n$,
and
for {\it every point}  $x\in D$ there is a sequence of
triangles $T_{x,0}\in \nabla_0,\,\,\,\, T_{x,1}\in \nabla_1,
\dots$ whose intersection
is the singleton $\{x\}$.
 }
\end{remark}

\end{document}